\documentclass[final]{dmtcs-episciences}
\usepackage[utf8]{inputenc}
\usepackage[round]{natbib}

\usepackage{amssymb,amsthm,amsmath}
\usepackage[mathcal]{euscript}
\usepackage[T1]{fontenc}
\usepackage{xspace}
\usepackage{microtype}
\usepackage{pdfsync}
\usepackage{float}
\usepackage[active]{srcltx}
\usepackage{tikz}
\usetikzlibrary{trees,arrows,shapes,snakes,fit,shadows,decorations.text,decorations.pathmorphing,calc}
\tikzstyle{implies}=[line width=0.7pt,->]
\tikzstyle{equiv}=[line width=0.7pt,<->]
\usepackage{hyperref}

\newtheorem{Thm}{Theorem}[section]
\newtheorem{Prop}[Thm]{Proposition}
\newtheorem{Lemma}[Thm]{Lemma}
\newtheorem{Cor}[Thm]{Corollary}

\newtheorem*{definition*}{Definition}
\newtheorem{remark}{Remark}

\usepackage{enumerate}

\numberwithin{equation}{section}

\DeclareSymbolFont{bbold}{U}{bbold}{m}{n}
\DeclareSymbolFontAlphabet{\mathbbold}{bbold}
\newcommand\um{\ensuremath{\mathbbold{1}}}

\def\pv#1{\ensuremath{{\sf#1}}}
\def\eval#1#2{[#1]_{\pv #2}}
\def\obn{\ensuremath{\overline{\pv B_{\mathit{n}}}}\xspace}
\def\Cl#1{\ensuremath{\mathcal #1}}
\def\Om#1#2{\ensuremath{\overline\Omega_{#1}{\sf#2}}}

\def\omup#1#2#3{\ensuremath{\Omega^{#1}_{#2}{\sf{#3}}}}
\def\oms#1#2{\omup{\sigma}{#1}{#2}}
\def\omc#1#2{\omup{\kappa}{#1}{#2}}
\def\omu#1#2{\omup{\um}{#1}{#2}}
\def\rk#1{\ensuremath{\mathop{\rm rank}({#1})}}
\def\rks#1#2{\ensuremath{\mathop{\rm rank}_{#1}({#2})}}
\def\pj#1{\ensuremath{p_{\sf#1}}}
\def\tcl#1{\ensuremath{\mathrm{cl}(#1)}}
\def\tcls#1{\ensuremath{\overline{#1}}}
\def\tclsv#1#2{\ensuremath{\overline{#2}}}
\def\tclin#1#2{\ensuremath{\mathrm{cl}_{#1}(#2)}}
\def\acls#1{\ensuremath{\langle #1\rangle_\sigma}}
\def\kacls#1{\ensuremath{\langle #1\rangle_\kappa}}

\def\weight{\ensuremath{\mathsf{w}}}
\let\le\leqslant
\let\leq\leqslant
\let\ge\geqslant
\let\geq\geqslant

\newcommand\initial{\ensuremath{\,\hat{\phantom{\ }}\,}}
\newcommand\final{\ensuremath{\$}} 
\newcommand\Approx{\ensuremath{\xi}}
\newcommand\sh{\ensuremath{\mathit{sh}}}

\author{%
  J. Almeida\affiliationmark{1}\thanks{Work partially supported by
    CMUP (UID/MAT/ 00144/2013), which is funded by FCT (Portugal) with
    national (MCTES) and European structural funds (FEDER), under the
    partnership agreement PT2020.}%
  \and J. C. Costa\affiliationmark{2}\thanks{Work supported, in part,
    by the European Regional Development Fund, through the program
    COMPETE, and by the Portuguese Government through FCT, under the
    project PEst-OE/MAT/UI0013/2014.}%
  \and M. Zeitoun\affiliationmark{3}\thanks{Work carried out with
    financial support from the French State, managed by the French
    National Research Agency (ANR) in the frame of ANR 2010 BLAN 0202
    01 FREC and of the "Investments for the future" Programme IdEx
    Bordeaux -- CPU (ANR-10-IDEX-03-02).}}

\title{Factoriality and the Pin-Reutenauer procedure}

\affiliation{Centro de Matem\'atica e Departamento de Matem\'atica,
  Faculdade de Ci\^encias,
  Universidade do Porto, Portugal\\
  Centro de Matem\'atica e Departamento de Matem\'atica e Aplica\c
  c\~oes, Universidade do Minho, Portugal\\
  Univ. Bordeaux, LaBRI, UMR 5800, F-33400 Talence, France}


\received{2015-6-4}

\accepted{2016-2-2}

\begin{document}
\publicationdetails{18}{2016}{3}{1}{650}
\maketitle

\begin{abstract}
  We consider implicit signatures over finite semigroups determined by
  sets of pseudonatural numbers. We prove that, under relatively simple
  hypotheses on a pseudovariety \pv V of semigroups, the finitely
  generated free algebra for the largest such signature is closed under
  taking factors within the free pro-\pv V semigroup on the same set of
  generators. Furthermore, we show that the natural analogue of the
  Pin-Reutenauer descriptive procedure for the closure of a rational
  language in the free group with respect to the profinite topology
  holds for the pseudovariety of all finite semigroups. As an
  application, we establish that a pseudovariety enjoys this property if
  and only if it is full. 
\end{abstract}

\section{Introduction}
\label{sec:introduction}

\noindent\textbf{Context and motivations.} This paper deals with the computation of the closure of a given rational language within a
relatively free algebra, with respect to a suitable implicit signature and a
profinite topology. A motivation for this line of research is the
\emph{separation problem}, which, given two rational languages~$K$ and~$L$,
asks whether there is a rational language from a fixed class $\mathcal{C}$
containing~$K$ and disjoint from~$L$. The separation problem has several
motivations. First, the membership problem for $\mathcal{C}$ reduces to the
separation problem for $\mathcal{C}$, since a language belongs to the class
$\mathcal{C}$ if and only if it is separable from its complement by a language
from~$\mathcal{C}$.

Furthermore, solving this problem gives more information about the class under
investigation, and is more robust when applying transformations to the class.
For instance, is was proved by \cite{Steinberg:delay-pointlikes:2001} and
\cite{PZplusone:2015} that the classical operator $\pv V\mapsto \pv V*\pv D$
on pseudovarieties preserves decidability of the separation problem, while it
has been shown by \cite{DBLP:journals/ijac/Auinger10} that it does not
preserve decidability of the membership problem (on the other hand, the status
with respect to separation is unknown for other operators that do not preserve
the decidability of membership, such as the power, as shown by
\cite{Auinger&Steinberg:2003}).

Finally, deciding separation for some class can be used to decide membership
for more involved classes: this is for instance a generic result in the
quantifier alternation hierarchy, established by \cite{PZ:icalp14}, that
deciding separation at level $\Sigma_n$ in this hierarchy entails a decision
procedure for membership at level $\Sigma_{n+1}$.

\cite{Almeida:1996d} has related the separation problem with a purely
topological question, which
is the main topic of this paper: the separation
problem has a negative answer on an instance $K,L$ of rational languages if and only if the
closures of~$K$ and~$L$ in a suitable relatively free profinite semigroup, which depends on
the class of separator languages we started from, have a nonempty intersection. Determining whether
such closures intersect can be in turn reformulated in terms of computation of pointlike
two-element sets in a given semigroup.

\smallskip Deciding whether closures of rational languages intersect is often nontrivial, in
particular because the profinite semigroup in question is uncountable in general. Yet,
several classes of languages enjoy a property called \emph{reducibility}\footnote{More
  precisely and technically, reducibility for 2-pointlike sets.} that states that the
closures of two rational languages intersect in the suitable relatively free profinite
semigroup if and only if their traces in a more manageable universe also intersect. This more
manageable universe may in particular be countable, and is therefore amenable to algorithmic
treatment. In summary, reducibility is a property of the class of separators under investigation
(or of the class of semigroups recognizing these separators), which reduces the search
of a witness in the intersection into a simpler~universe.

\smallskip
The most important example from the historical point of view is the class of languages
recognized by finite groups. In this case, the relatively free algebra
is the
free group over some set $X$ of generators, which is indisputably much better understood than
the free profinite group over~$X$. In particular, it is countable. Since it is known that the closures in
the free profinite group of two rational languages intersect if and only if their traces in the
free group also intersect (that is, the class of finite groups enjoys the reducibility
property), this justifies the quest for an algorithm computing the closure in the free group of
a rational language. Such an algorithm is known as the Pin-Reutenauer procedure, which we
describe below, and has been developed along a successful line of
research, see the work of~\cite{Pin&Reutenauer:1991,Pin:1991,Ash:1991,Henckell&Margolis&Pin&Rhodes:1991,Ribes&Zalesskii:1993a,Herwig&Lascar:2000,Auinger:2004,Auinger&Steinberg:2005b}. As
a consequence, the separation problem by group languages is decidable.

\smallskip This framework can be generalized to classes consiting of
other types of semigroups than just groups.
Denote by $\kappa$ the signature consisting of the binary multiplication and
the unary $(\omega-1)$-power, with their usual interpretation in profinite
semigroups. Note that the set of all $\kappa$-terms over $X$ is isomorphic to
the free group over $X$: the mapping sending each generator to itself and
$x^{\omega-1}$ to $x^{-1}$, the inverse of $x$ in the free group, can be
extended to a group isomorphism. More generally, given a pseudovariety \pv V
of finite semigroups, consider the semigroup $\omc XV$, which can be seen as
the set of all interpretations over \pv V of $\kappa$-terms over~$X$. This
subalgebra of the pro-\pv V semigroup over $X$ is countable and thus, as said
above, amenable to algorithmic treatment. One central problem in this
context is the $\kappa$-word problem: given two $\kappa$-terms over~$X$,
decide whether they represent the same element in the relatively free
profinite semigroup in the pseudovariety under consideration. This problem has
already been investigated for several classical pseudovarieties besides that
of finite groups, for instance by \cite{Almeida:J:1991,Almeida:1994a} for the
pseudovariety \pv J of all finite $\mathcal{J}$-trivial semigroups,
by~\cite{Almeida&Zeitoun:2004:sal,Almeida&Zeitoun:AutomTheorApproac:2007} for
the pseudovariety \pv R of all finite $\mathcal{R}$-trivial semigroups,
by~\cite{Costa:2001} for the pseudovariety \pv{LSl} of all finite semigroups
whose local monoids are semilattices and by \cite{Moura:2009b} for the
pseudovariety \pv{DA} of all finite semigroups whose regular
$\mathcal{J}$-classes are aperiodic semigroups. Moreover, reducibility has
been shown to hold for several pseudovarieties, in particular by
\cite{Almeida:2002a} for \pv J, by \cite{Almeida&Costa&Zeitoun:2004} for \pv
R, by~\cite{Costa&Teixeira:Tameness-pseudovariety-LSl:2004} for \pv{LSl} and,
as already mentioned, by \cite{Ash:1991,Almeida&Steinberg:2000a} for 
the pseudovariety \pv G of all finite groups. A further example is the
pseudovariety \pv A of aperiodic
languages which, in a forthcoming paper, will be derived from the work of
\cite{Henckell:1988}, recently revisited
by \cite{PZ:lics14,PZ:FOfull16}, from which one can derive reducibility
of this class. In other words, for these classes of languages, the separation
problem reduces to testing that the intersection of the closures of two given
rational languages in the suitable countable relatively free algebra is empty. This
motivates designing algorithms to compute closures of rational languages in
these relatively free algebras. This is one of the main contributions of this~paper.

\medskip\noindent\textbf{The Pin-Reutenauer procedure.}
In the core of the paper, we investigate how the profinite closure of rational languages in
free unary algebras interacts with concatenation and iteration. The natural guide for this
work is provided by a procedure proposed by \cite{Pin&Reutenauer:1991} for
the case of the free group. This procedure gives a way to compute a representation of
the closure of a rational language inductively on the structure of the rational expression. Of
course, the closure of a union is the union of the closures. The other two rules of the
Pin-Reutenauer procedure deal with concatenation and iteration. For instance, when computing
in \omc XV, the smallest subalgebra of the pro-\pv V semigroup closed under multiplication
and $(\omega-1)$-power, establishing the Pin-Reutenauer procedure amounts to showing the
following equalities:
\begin{align*}
  \tclsv V{KL} &= \tclsv V{K}\,\tclsv V{L},\\
  \tclsv V{L^+} &= \kacls{\tclsv VL},
\end{align*}
where $\tclsv VL$ is the topological closure of $L$ in \omc XV, and $\kacls L$ is the subalgebra of \omc
XV generated by $L$.  Notice that these equalities yield a recursive procedure to compute a
finite algebraic representation of $\tclsv V L$ when $L$ is rational.  Such a finite
representation may not immediately yield algorithms to decide membership in $\tclsv VL$ for a
given rational language $L$, but it reduces the problem of computing topological closures
$\tclsv VL$ to the problem of computing algebraic closures $\kacls L$. Since the signature
$\kappa$ is finite, this representation also provides a recursive enumeration of elements of $\tclsv
VL$. Additionally, assume that the following two properties hold:
\begin{enumerate}[(1)]
\item the word problem for $\kappa$-terms over \pv V is decidable,
\item the pseudovariety \pv V is $\kappa$-reducible.
\end{enumerate}
Then one can decide the separation problem of two rational languages $K,L$ by a \pv
V-recognizable language. Indeed, \cite{Almeida:1996d} has shown that this problem is equivalent to checking whether
the closures of $K$ and $L$ in \Om XV intersect, which by reducibility is equivalent to
checking whether $\tclsv VK\cap \tclsv VL\not=\emptyset$. In turn, this may be tested by
running two semi-algorithms in parallel:
\begin{enumerate}[(1)]
\item one that enumerates elements of $\tclsv VK$ and $\tclsv VL$ and
  checks, using the solution over \pv V of the word
  problem for $\kappa$-terms, whether there is some common element;
\item another one that enumerates all potential \pv V-recognizable separators.
\end{enumerate}
Thus, the Pin-Reutenauer procedure is one of the ingredients to understand why a given class
has decidable separation problem.

\medskip\noindent\textbf{Contributions.} It has been established recently by
\cite{ACZ:Closures:14} that The Pin-Reutenauer procedure holds for a number of
pseudovarieties. However, the results of this paper rely on independent,
technically nontrivial results for the pseudovariety \pv A of aperiodic
semigroups: first, it was proved that the Pin-Reutenauer procedure is valid
for \pv A using the solution of the word problem for the free aperiodic
$\kappa$-algebra given by
\cite{McCammond:1999a,HuschenbettKufleitner:omega:STACS14,Almeida&Costa&Zeitoun:normal_forms}.
Then, a transfer result was established to show that it is also valid for
subpseudovarieties of~\pv A.

\smallskip
In this paper, we revisit the Pin-Reutenauer procedure, obtaining general results with
simpler arguments. We consider unary signatures, made of multiplication and operations of
arity 1. Our main result, Theorem~\ref{t:main}, establishes that the Pin-Reutenauer procedure
holds for the pseudovariety \pv S of all finite semigroups, for unary signatures satisfying
an additional technical condition, which is met for~$\kappa$. The fact that rational
languages are involved is crucial, since, as observed by~\citet[p.~10]{ACZ:Closures:14},
the equality $\tcls{KL}=\tcls K\cdot\tcls L$ fails for some languages $K,L\subseteq
X^+$, where closures are taken with respect to~\pv S and the 
signature $\kappa$.

This result is obtained by first investigating a property named
\emph{factoriality}. Factoriality of \pv V with respect to, say, the signature $\kappa$ means that \omc
XV is closed under taking factors in \Om XV. It was shown by \cite{ACZ:Closures:14} that if
\pv V is factorial, then the Pin-Reutenauer procedure holds with respect to\ concatenation, that is,
$\tclsv V{KL} = \tclsv V{K}\,\tclsv V{L}$ for arbitrary $K,L$ (not just for rational
ones). However, it was also noted that the pseudovariety \pv S cannot be factorial for
nontrivial countable signatures, such as $\kappa$. In contrast, we show that any nontrivial
pseudovariety of semigroups~\pv V closed under concatenation is factorial for the signature
\um{} consisting of multiplication and \emph{all} unary operations. As an application, we
obtain a new proof that the minimum ideal of the free pro-\pv V semigroup on at least two
generators contains no \um-word. This property is a weaker version of a result
obtained by \cite{Almeida&Volkov:2002b}. Besides the  independent
interest of such results, the technical tool used to prove them, named \emph{factorization history},
is also the key to establish that the Pin-Reutenauer is valid for \pv S. We further characterize
pseudovarieties in which the Pin-Reutenauer procedure holds in terms of an abstract property
named fullness, introduced by \cite{Almeida&Steinberg:2000a}.  The main idea is that the validity
of the Pin-Reutenauer procedure for a pseudovariety \pv V is inherited by a
subpseudovariety \pv W, as established by \cite{ACZ:Closures:14}, provided both \pv V and \pv W are full.
Conversely, we prove that if the Pin-Reutenauer procedure works for \pv V, then \pv V is full.  Since
the pseudovariety of all finite semigroups is full, this yields that a pseudovariety enjoys
the Pin-Reutenauer property if and only if it is full.

Finally, we show that a variation of the Pin-Reutenauer procedure, known to hold in
the case of all groups, also holds for pseudovarieties of groups in which every finitely generated
subgroup of the free $\kappa$-algebra is~closed. 

\medskip\noindent\textbf{Organization.} The paper is organized as follows. In Section~\ref{sec:um-factoriality}, we
introduce the notion of history of a factorization and we show that any nontrivial pseudovariety closed under
concatenation product is \um-factorial. In Section~\ref{sec:closures},
we establish that the Pin-Reutenauer property holds for
\pv S and unary signatures satisfying an additional condition. In Section~\ref{sec:PR}, we relate the Pin-Reutenauer
property with fullness, in the general case and in the case of pseudovarieties of groups.

\section{\texorpdfstring{\um}1-factoriality}
\label{sec:um-factoriality}

We assume that the reader has some familiarity with profinite semigroups. For
details, we refer the reader to the books of
\cite{Almeida:1994a,Rhodes&Steinberg:2009qt} and to the article of
\cite{Almeida:2003c}. Here, we briefly introduce the required notation and
key~notions.

\medskip\paragraph{\textbf{Preliminaries.}}
Throughout the paper, we work with a finite alphabet $X$. For a
pseudovariety \pv V of semigroups, we denote by $\Om XV$ the free
pro-\pv V semigroup generated by $X$. Elements of $\Om XV$ are called
\emph{$X$-ary implicit operations} over~\pv V. See the paper of \cite{Almeida:1994a} for details.

An \emph{implicit signature}, as defined by \cite{Almeida&Steinberg:2000a}, is a set of
implicit operations of finite arity including the formal binary
multiplication. A \emph{$\sigma$-semigroup} is an algebra in the
signature $\sigma$ whose multiplication is associative.  Thus,
$\sigma$-semigroups form a Birkhoff variety. We call an
element of the free $\sigma$-semigroup generated by $X$ a
\emph{$\sigma$-term}. For convenience, we allow the empty~$\sigma$-term.

Every pro-\pv V semigroup has a natural structure of
$\sigma$-semigroup. We denote by $\oms XV$ the sub-$\sigma$-semigroup
of \Om XV generated by $X$. A \emph{$\sigma$-word over \pv V} is an element of
$\oms XV$. We denote by $\eval \_V$ the surjective homomorphism of
$\sigma$-semigroups that associates to a $\sigma$-term $t$ its
interpretation $\eval tV$ in $\oms XV$. When $t$ is a word and $\pv V$ is clear
from the context, we write $t$ instead of~$\eval t V$.

\medskip
\paragraph{\textbf{Unary implicit signatures.}}
Let $\widehat{\mathbb{N}}$ be the profinite completion of~$(\mathbb{N},+)$,
\emph{i.e.}, the free profinite monoid on one generator. We denote by \um\ the
implicit signature consisting of multiplication together with all implicit operations
$\_\vphantom{|}^\alpha$ with $\alpha\in\widehat{\mathbb{N}}\setminus\mathbb{N}$. An implicit
signature is called \emph{unary} if it is contained in~\um\ and it contains at least one
unary implicit operation.  For a unary implicit signature~$\sigma$, an element
$\alpha\in\widehat{\mathbb{N}}$ such that the $\alpha$-power operation
$\_\vphantom{|}^\alpha$ belongs to~$\sigma$ is said to be a \emph{$\sigma$-exponent}. Note
that by definition of \um, every $\sigma$-exponent is infinite. An important example of a
unary implicit signature is the signature $\kappa$, for which $\omega-1$ is the only
$\kappa$-exponent.

The \emph{$\sigma$-rank} $\rks\sigma t$ of a $\sigma$-term $t$ is the
maximal nesting depth of elements of $\sigma$, disregarding multiplication,
that occur in $t$. It is defined inductively by
$\rks\sigma{t_1t_2}=\max(\rks\sigma {t_1},\rks\sigma {t_2})$ and
$\rks\sigma{\pi(t_1,\ldots, t_n)}=1+\max_{1,\ldots ,n}(\rks\sigma{t_i})$
in case $\pi$ is an operation from $\sigma$ which is not multiplication.
For a $\sigma$-term
\begin{equation}
  \label{eq:t}
  t=t^{}_0s_1^{\alpha_1}t^{}_1\cdots s_m^{\alpha_m}t^{}_m,
\end{equation}
where the $t_i$'s and the $s_j$'s are $\sigma$-terms such that
$\rks\sigma{t_i}\leq\rks\sigma {s_j}=\rks\sigma t-1$ and each $\alpha_j$ is a
$\sigma$-exponent, we denote by $\nu_\sigma(t)$ the number $m$ of
subterms $s_i^{\alpha_i}$ of $t$. When $\sigma$ is clear from the
context, we may write $\rk t$, $ \nu(t)$ instead of
$\rks\sigma t$, $\nu_\sigma(t)$, respectively.

\medskip
\paragraph{\textbf{Complete unary implicit signatures.}}
A unary implicit signature $\sigma$ is said to be \emph{complete} if the set of
$\sigma$-exponents is stable under the mappings $\alpha\mapsto\alpha-1$ and
$\alpha\mapsto\alpha+1$. Note that \um~is complete, while $\kappa$ is not. The intersection of a nonempty
set of complete unary signatures either consists of multiplication solely, or is again a
complete unary signature.
Therefore, the smallest complete unary signature containing a given unary signature $\sigma$
exists. It is called the \emph{completion} of~$\sigma$ and it is denoted by~$\bar\sigma$. By
definition, we have $\sigma\subseteq\bar\sigma$, and a signature $\sigma$ is complete if and
only if $\bar\sigma=\sigma$. Note that  for every \um-exponent~$\alpha$ and every $u\in\Om XS$, the equalities
$u^{\alpha-1}=u^\alpha u^{\omega-1}$ and $u^{\alpha+1}=u^\alpha u$ hold. This proves the following
useful fact.

\begin{remark}
  \label{rem:inclusive}
  Let $\sigma$ be a  unary signature containing $\kappa$. Then $\omup{\bar\sigma}XV=\oms XV$.
\end{remark}

\subsection{Factorization sequences}
\label{sec:fact-sequ}

For $\alpha\in\widehat{\mathbb N}$, we choose a sequence
$(\Approx_n(\alpha))_n$ of natural integers converging to $\alpha$. One
can assume that $(\Approx_n(\alpha))_n$ is constant if $\alpha$ is finite,
or strictly increasing otherwise. Let~$t$ be a \um-term. We denote by
$\Approx_n(t)$ the word obtained by replacing each subterm $v^\alpha$
with $\alpha$ infinite by $v^{\Approx_n(\alpha)}$, recursively. For
instance, $\Approx_n((a^\alpha
b)^\beta)=(a^{\Approx_n(\alpha)}b)^{\Approx_n(\beta)}$.  The
factorizations $\Approx_n(t)=x\cdot y$ with $x\in X^*$ and $y\in
X^+$ may be obtained recursively as follows:
\begin{itemize}
\item if $\rk t=0$, then $\Approx_n(t)=t$ for all $n$ and there are
  $|t|$ such factorizations of~$\Approx_n(t)$;
\item if $\rk t>0$ and 
  $t=t^{}_0s_1^{\alpha_1}t^{}_1\cdots s_m^{\alpha_m}t^{}_m$, where the $t_i$'s
  and the $s_j$'s are $\sigma$-terms such that
  $\rk{t_i}\leq\rk{s_j}=\rk{t}-1$ (where the $t_i$'s may be empty), then the factorizations of
  $\Approx_n(t)$ are those of the following forms:
  \begin{align}
    \label{item:factorization-1}
      \Approx_n(t)&=       \Approx_n(t_0s_1^{\alpha_1}\cdots t_{j-1}s_j^{\alpha_j})\,t_j'
      \cdot       t''_j\,\Approx_n(s_{j+1}^{\alpha_{j+1}}t_{j+1}\cdots s_m^{\alpha_m}t_m)     \\
\noalign{where $\Approx_n(t_j)=t'_jt''_j$, and}
    \label{item:factorization-2}
    \Approx_n(t)&=     \Approx_n(t_0s_1^{\alpha_1}\cdots
      s_{j-1}^{\alpha_{j-1}}t_{j-1}s_j^k)s'_j
      \cdot       s''_j\Approx_n(s_j^\ell t_js_{j+1}^{\alpha_{j+1}}\cdots
      s_m^{\alpha_m}t_m)   \end{align}
  where $\Approx_n(s_j)=s'_js''_j$, 
  $k,\ell\in\mathbb{N}$, and $k+\ell+1=\Approx_n(\alpha_j)$.
\end{itemize}
The condition $y\in X^+$, forbidding $y$ to be empty, is used recursively
to ensure that each factorization of $\Approx_n(t)$ is either of type
\eqref{item:factorization-1} or \eqref{item:factorization-2}, but not of
both types: one can verify that each factorization of $\Approx_n(t)$ is
obtained by exactly one of the equations \eqref{item:factorization-1}
and \eqref{item:factorization-2}, where $j,t'_j,t''_j$ (in case
\eqref{item:factorization-1}), or $j,k,\ell,s'_j,s''_j$ (in case
\eqref{item:factorization-2}) are uniquely determined. In particular, the
factorization
\[
  \Approx_n(t_0s_1^{\alpha_1}\cdots
  t_{p-1}s_p^{\alpha_p}t_p)\cdot\,\Approx_n(s_{p+1}^{\alpha_{p+1}}t_{p+1}\cdots
  s_m^{\alpha_m}t_m)
\]
cannot be of type~\eqref{item:factorization-1},
since this would force $t''_p$ to be empty, which is forbidden. This
factorization is in fact of type~\eqref{item:factorization-2} with
$j=p+1$ and $k=0$.

As an example, for $t=a^\omega ba^\omega$, the expression \eqref{eq:t} is
obtained for $m=2$, where $t_0$ and $t_2$ are empty, while
$s_1^{\alpha_1}=a^\omega$, $t_1=b$, and
$s_2^{\alpha_2}=a^\omega$. Assuming $\Approx_n(\omega)=n!$, we obtain
\begin{itemize}
\item[--] the factorization $a^{n!}\cdot ba^{n!}$ 
  by~\eqref{item:factorization-1}, with $j=1$, $t'_1$ empty and $t''_1=b$;
\item[--] the factorization $a^{n!}b\cdot a^{n!}$ 
  by~\eqref{item:factorization-2}, $j=2$, $k=0$, $\ell=n!-1$, $s'_2$ empty and $s''_2=a$.
\end{itemize}

\medskip
The \emph{history} $h_n(t,x,y)$ of a factorization
$\Approx_n(t)=xy$ is defined recursively as follows:
\begin{itemize}
\item[--] if $\rk{t}=0$, then $h_n(t,x,y)=(x,y)$;
\item[--] if $\rk{t}>0$ and the factorization is of the
  form~\eqref{item:factorization-1}, then $h_n(t,x,y)$ is obtained
  by concatenating the pair $(1,j)$ with $h_n(t^{}_j,t'_j,t''_j)$;
\item[--] if $\rk{t}>0$ and the factorization is of the
  form~\eqref{item:factorization-2}, then $h_n(t,x,y)$ is obtained 
  by concatenating the 4-tuple $(2,j,k,\ell)$ with
  $h_n(s^{}_j,s'_j,s''_j)$.
\end{itemize}
The history $h_n(t,x,y)$ is thus a word on an alphabet that depends on the
integer~$n$, which gives information on how the word $\Approx_n(t)$ is
split by the factorization $xy$. Note that the length of the history $h_n(t,x,y)$ is at most
$\rk{t}+1$.

The \emph{simplified history} $\sh_n(t,x,y)$ of the factorization
$\Approx_n(t)=xy$ is obtained from the history $h_n(t,x,y)$ by replacing
each 4-tuple $(2,j,k,\ell)$ by $(2,j)$.  On the other hand, dropping the
first two components of each letter of the history $h_n(t,x,y)$, we
obtain a word whose letters are pairs of nonnegative integers, which we
identify with an integer vector in even dimension, called the
\emph{exponent vector}. A factorization of~$\Approx_n(t)$ can be
recovered from its history but may not be recoverable from its
simplified history without the extra information contained in the
exponent vector.

\medskip
Observe that $(\Approx_n(t))_n$ converges to $\eval t S$ in $\Om XS$.
We will be interested in sequences $(x_n,y_n)_n$ such that
$x_ny_n=\Approx_n(t)$. We call such a sequence a \emph{factorization
  sequence for~$t$}.

\medskip It will be convenient in the proofs to work with factorization sequences having
additional properties. Note that the set of simplified histories of factorizations
of~$\Approx_n(t)$ is finite and depends only on~$t$. Moreover, the dimension of all exponent
vectors is bounded by $2\rk t$. Therefore, any factorization sequence for $t$ has a
subsequence whose
\begin{enumerate}[$(a)$]
\item induced sequence of simplified histories is constant,
\item induced sequence of exponent vectors belongs to $\mathbb{N}^d$ for
  some constant $d$ and
  converges in~$\widehat{\mathbb{N}}^d\setminus\mathbb{N}^d$.
\end{enumerate}
We call \emph{filtered} a sequence with these properties. An
application of this notion is the following simple statement.

\begin{Lemma}
  \label{lem:converging-factorization-sequences}
  Let $(x_n,y_n)_n$ be a factorization sequence for a \um-term. Then both
  $(x_n)_n$ and $(y_n)_n$ have subsequences converging in \Om XS to
  $\um$-words.
\end{Lemma}

\begin{proof}
  Let $t$ be the \um-term of the statement. By the above, one may assume that the sequence $(x_n,y_n)_n$ is
  filtered. We proceed by induction on $\rk{t}$. The case $\rk{t}=0$ is straightforward. Otherwise,
  let $t=t^{}_0s_1^{\alpha_1}t^{}_1\cdots s_m^{\alpha_m}t^{}_m$ as in \eqref{eq:t}. There are two
  cases, according to the first letter of $\sh_n(t,x,y)$, which can be of the form $(1,j)$ or
  $(2,j)$. Both cases are similar, so assume that it is of the form $(1,j)$. Therefore, the
  factorization $x_{n}y_{n}$ of $\Approx_{n}(t)$ is given by \eqref{item:factorization-1}, hence
  $x_{n}=\Approx_{n}(t^{}_0s_1^{\alpha_1}\cdots t^{}_{j-1}s_j^{\alpha_j})\,t'_{j,n}$ and
  $y_{n}=t''_{j,n}\,\Approx_n(s_{j+1}^{\alpha_{j+1}}t^{}_{j+1}\cdots s_m^{\alpha_m}t^{}_m)$ where
  $\Approx_n(t^{}_j)=t'_{j,n}t''_{j,n}$. By definition, $t_j$ is a \um-term and $\rk{t_j}<\rk{t}$,
  whence by induction $(t'_{j,n})_n$ and $(t''_{j,n})^{}_n$ have subsequences converging to \um-terms,
  respectively $t'_j$ and $t''_j$. Therefore, $(x_n)_n$ (resp.~$(y_n)_n$) has a subsequence converging to
  the \um-term $t^{}_0s_1^{\alpha_1}\cdots t^{}_{j-1}s_j^{\alpha_j}\cdot t'_{j}$ (resp.~$t''_j\cdot
  s_{j+1}^{\alpha_{j+1}}t^{}_{j+1}\cdots s_m^{\alpha_m}t^{}_m$).
\end{proof}

\subsection{Factoriality of some pseudovarieties}
\label{sec:fact-some-pseud}

A pseudovariety of semigroups is said to be \emph{closed under concatenation} if the
corresponding variety of rational languages has that property. A nontrivial pseudovariety \pv
V is closed under concatenation if and only if it contains \pv A, the pseudovariety of
aperiodic (or group-free) semigroups, and the multiplication of the profinite semigroup \Om
XV is an open mapping for every finite
alphabet~$X$ as proved by~\cite{Almeida&ACosta:2007a} based on results
of~\cite{Straubing:Aperiodic-homomorphisms-concatenation-product:1979:a}
(in the monoid case)
and~\cite{Chaubard&Pin&Straubing:Actions-wreath-products-C-varieties:2006:a}
(in the semigroup case) characterizing such pseudovarieties in
terms of certain algebraic closure properties.

\medskip A pseudovariety \pv V is said \emph{$\sigma$-factorial} if,
for every finite alphabet $X$, every factor in $\Om XV$ of a
$\sigma$-word over \pv V is also a $\sigma$-word over \pv V. Note that the
pseudovariety~\pv S is not $\kappa$-factorial, since $x^\alpha$ is a
prefix of $x^\omega$ for every $\alpha\in\widehat{\mathbb{N}}$.

\begin{Thm}
  \label{t:S-1-factorial}
  Let \pv V be a pseudovariety closed under concatenation. Then~\pv V is \um-factorial.
\end{Thm}

\begin{proof}
  The statement is obvious if \pv V is trivial. Otherwise, let $u=vw$ be a factorization
  in~\Om XV of an arbitrary element of~\omu XV. Let $t$ be a \um-term such that $\eval t
  V=u$. Since the sequence $(\Approx_n(t))_n$ converges to $u=vw$ in \Om XV and the
  multiplication is open in~\Om XV, for all sufficiently large $n$, each $\Approx_n(t)$ may
  be factorized as $\Approx_n(t)=v_nw_n$ in such a way that $\lim v_n=v$ and $\lim w_n=w$.

  By Lemma~\ref{lem:converging-factorization-sequences}, both $(v_n)_n$
  and $(w_n)_n$ have  subsequences converging, in $\Om XS$, to \um-words over \pv S.
  Therefore, in $\Om XV$, these subsequences  converge to  \um-words
  over \pv V, so that $v$ and $w$ are actually $\um$-words over $\pv V$.
\end{proof}

For a pseudovariety \pv H of groups, $\overline{\pv H}$ denotes the pseudovariety of all
finite semigroups whose subgroups lie in~\pv H.  In particular, when \pv H is the trivial
pseudovariety, then $\overline{\pv H}=\pv A$. It is a well-known and elementary fact that
$\overline{\pv H}$ is always closed under concatenation. Denote by $\pv B_n$ the Burnside
pseudovariety of all finite groups of exponent dividing $n$. The pseudovariety \obn
 is thus defined by the pseudoidentity $x^{\omega+n}=x^\omega$.

In the following result, the special case $n=1$, corresponding to the pseudovariety \pv A,
was first shown by~\cite{Almeida&Costa&Zeitoun:normal_forms} with a much more involved
proof.

\begin{Cor}
   \label{c:Bn-k-factorial}
   For every positive integer $n$ the pseudovariety \obn is $\kappa$-factorial. In
   particular, the pseudovariety \pv A is $\kappa$-factorial.
\end{Cor}

\begin{proof}
  We claim that the equality $\omc X\obn=\omu X\obn$ holds for $|X|=1$
  and so also for every finite alphabet $X$. Therefore, the result
  follows from Theorem~\ref{t:S-1-factorial}. To prove the claim, we
  show that
  $\Om {\{x\}}{\obn}=\{x^k\mid
  k\in\mathbb{N}\}\cup\{x^\omega,x^{\omega+1},\ldots,x^{\omega+(n-1)}\}$.
  For this, let $\alpha$ be a \um-exponent and let $(a_k)_k$ be a
  sequence of integers converging to $\alpha$. One can assume that 
  $a_k$ modulo $n$ is a constant $a$, hence $a_k=nb_k+a$ with
  $b_k\in\mathbb{N}$ for all $k$. In $\Om {\{x\}}\obn$, we then have
  $x^\alpha=x^{\omega+\alpha}=\lim_kx^{\omega+a_k}=\lim_kx^{\omega+nb_k+a}=x^{\omega+a}\in\omc
  X\obn$.
\end{proof}

Another application of Theorem~\ref{t:S-1-factorial} is the following
result, which is a weaker version of one that was established in
\cite[Corollary~8.12]{Almeida&Volkov:2002b}. Although the original
result was formulated for the pseudovariety of all finite semigroups,
the proof applies unchanged to pseudovarieties containing all finite
local semilattices. Related results, under the same hypothesis as the
following corollary, have been obtained by~\cite{Steinberg:2010}.

\begin{Cor}
  \label{c:JA-Volkov}
  If $|X|\geq2$ and \pv V is a nontrivial pseudovariety closed under
  concatenation, then there is no \um-word in the
  minimum ideal of~\Om XV.
\end{Cor}

\begin{proof}
  Since \pv V is \um-factorial by Theorem~\ref{t:S-1-factorial} and since every element of
  \Om XV is a factor of every element of the minimum ideal, if there were a \um-word in the
  minimum ideal then every element of~\Om XV would be a \um-word. We claim that this is
  impossible under the hypothesis that $|X|\geq2$.

  To prove the claim, observe that by definition, every \um-word of~\Om XV
  which is not a word has at least one infinite power of a finite word as an infix. In
  particular, it admits as factors powers of finite words of arbitrarily large
  exponent. Thus, it suffices to exhibit an element of~\Om XV that fails this condition. For
  this purpose let $x,y\in X$ be distinct letters and consider the Prouhet-Thue-Morse
  substitution, defined by $\varphi(x)=xy$, $\varphi(y)=yx$, and $\varphi(z)=z$ for all $z\in
  X\setminus\{x,y\}$. This extends to a unique continuous endomorphism of~\Om XV, which we
  also denote $\varphi$. Since, as proved by~\cite{Hunter:1983}, the
  monoid of continuous endomorphisms of~\Om XV is profinite
  under the pointwise convergence topology, we may consider the element
  $\varphi^\omega(x)=\lim\varphi^{n!}(x)$. Now, it is well known that each word
  $\varphi^{n!}(x)$ is cube free (see, for instance,
  \cite{Lothaire:1983}). Since \pv V is nontrivial and closed
  under concatenation product, it contains \pv A. Therefore, the sets of the form $(\Om
  XV)^1u(\Om XV)^1$, where $u$~is a word, are open \cite[Theorem~3.6.1]{Almeida:1994a}.
  Hence, $\varphi^\omega(x)$ is also free of cubes of finite words and so
  $\varphi^\omega(x)$~is not a \um-word.
\end{proof}

\section{The Pin-Reutenauer procedure over \pv S for pure signatures}
\label{sec:closures}

Given a pseudovariety \pv V of semigroups, an implicit signature $\sigma$ and a subset $L\subseteq\oms XV$, we denote by \tclsv VL the closure of $L$ in~\oms XV. Both the implicit signature $\sigma$ and the pseudovariety \pv V are understood in this notation. We are interested in computing a representation of such closures in two cases:
\begin{enumerate}[$(a)$]
\item \label{i:prw} when $L$ is of the form $\pj V(K)$ for some rational subset $K$ of $X^+$,
  where \pj V is the natural continuous homomorphism from $\Om XS$ to $\Om XV$;
\item \label{i:prs} when $L$ is a rational subset of $\oms XV$.
\end{enumerate}

Recall that the class of rational subsets of a semigroup $M$ is the smallest family of subsets of
$M$ containing the empty set and the singletons $\{m\}$ for $m\in M$, and closed under union $(Y,Z)\mapsto Y+Z$, product $(Y,Z)\mapsto YZ=\{yz\mid y\in Y,\ z\in Z\}$ and iteration $Y\mapsto Y^+=\bigcup_{k\geq1}Y^k$. Since the homomorphic image of a rational set is rational, any set of the form~\ref{i:prw} is also of the form~\ref{i:prs}. Conversely, there are of course rational sets of \oms XV that are not obtained as image of a rational set of $X^+$ under \pj V, such as the singletons $\{a^\alpha\}$ where $\alpha\in\widehat{\mathbb{N}}$ is a $\sigma$-exponent.

\smallskip

We say that the \emph{Pin-Reutenauer procedure} holds for a
class \Cl C of subsets of~\oms XV if, for every $K,L\in\Cl C$, the
following conditions are satisfied:
\begin{align}
  \label{eq:PR-product}
  \tclsv V{KL} &= \tclsv V{K}\,\tclsv V{L},\\
  \label{eq:PR-plus}
  \tclsv V{L^+} &= \acls{\tclsv VL},
\end{align}
where $\acls U$ denotes the $\sigma$-subalgebra generated by the subset
$U$ of $\oms XV$. Again, in this notation, the fact that closures are taken in $\oms XV$ is understood. 

\smallskip
We say that \pv V is \emph{(weakly) $\sigma$-PR} if, for every finite
alphabet $X$, the Pin-Reutenauer procedure holds for the class
of all subsets of~\oms XV of the form $\pj V(L)$ with $L\subseteq X^+$ a
rational language. We say that \pv V is \emph{strongly
  $\sigma$-PR} if, for every finite alphabet $X$, the Pin-Reutenauer
procedure holds for the class of all rational subsets of~\oms
XV.

In this section, we only deal with the pseudovariety \pv S. In
Section~\ref{sec:PR}, we shall transfer our results from \pv S to other
pseudovarieties. The main result of this section is the following
theorem. It applies only to pure signatures, which we describe below.

\begin{Thm}
  \label{t:main}
  The pseudovariety \pv S is $\sigma$-PR for every pure unary
  signature $\sigma$ containing $\kappa$.
\end{Thm}

The additional \emph{purity} property that $\sigma$ is required to
possess is the following.
\begin{definition*}
  A unary signature $\sigma$ is said to be \emph{pure} if, for every positive integer $d$
  and for all $\alpha\in\widehat{\mathbb{N}}$, if $d\alpha$ is a $\bar\sigma$-exponent, then
  $\alpha$ is also a $\bar\sigma$-exponent.
\end{definition*}
Note that the quotient of $d\alpha$ by $d$ is actually uniquely determined: if
$\alpha,\beta\in\widehat{\mathbb{N}}$ and $d\in\mathbb{N}\setminus\{0\}$ are such that
$d\alpha=d\beta$, then $\alpha=\beta$. This follows immediately from the
fact that the free profinite group on one generator, which is
isomorphic to $\widehat{\mathbb{N}}\setminus\mathbb{N}$, is torsion-free.
Let us show this fact directly: $d\alpha=d\beta$ means that all finite
semigroups satisfy $x^{d\alpha}=x^{d\beta}$. To show that all finite semigroups also satisfy
$x^{\alpha}=x^{\beta}$, it is sufficient to consider 1-generated semigroups. Such
semigroups have presentations of the form $S_{m,p}=\langle a:a^m=a^{m+p} \rangle$, for
integers $m,p\geq0$. Note that the semigroup homomorphism $\varphi:S_{m,p}\to S_{dm,dp}$ mapping
$a$ to $a^d$ is injective. Since $S_{dm,dp}$ satisfies $x^{d\alpha}=x^{d\beta}$, we have in
$S_{dm,dp}$ the equalities $\varphi(a^\alpha)=a^{d\alpha}=a^{d\beta}=\varphi(a^\beta)$,
whence $S_{m,p}$ satisfies $x^\alpha=x^\beta$. This proves that~$\alpha=\beta$.

\medskip
In view of the following lemma, Theorem~\ref{t:main} can be applied to the signature~$\kappa$.
\begin{Lemma}
  \label{l:kappa-pure}
  The unary signature $\kappa$ is pure.
\end{Lemma}

\begin{proof}
  Every $\bar\kappa$-exponent is of the form $\omega+n$, where $n\in\mathbb{Z}$. Therefore,
  it suffices to show that, if $n$ is an integer, $d$ is a positive
  integer, and $\alpha\in\widehat{\mathbb{N}}$ is such that
  $\omega+n=d\alpha$, then $d$ divides $n$, whence
  $\alpha=\omega+\frac nd$ is again a $\bar\kappa$-exponent. For that purpose, consider the unique continuous
  homomorphism of additive monoids
  $\varphi:\widehat{\mathbb{N}}\to\mathbb{Z}/d\mathbb{Z}$ which maps
  $1$ to~$1$. We have $\varphi(\omega)=\varphi(\lim_kk!)=0$ and $\varphi(d\alpha)=d\varphi(\alpha)=0$, and we deduce from the equality $\omega+n=d\alpha$ that
  $\varphi(n)=0$. 
\end{proof}

To establish Theorem~\ref{t:main}, we first prove a technical key lemma in Section~\ref{sec:key-lemma}. We shall then
consider separately the cases of concatenation and iteration respectively in
Sections~\ref{sec:concatenation} and~\ref{sec:iteration}.

\subsection{A key lemma}
\label{sec:key-lemma}

We first prove a technical result which will be the key lemma in the sequel. It shows that,
under suitable hypotheses, one can balance the factors of a factorization of a given
$\bar\sigma$-term to make them $\bar\sigma$-terms themselves, without affecting membership in
given clopen sets. For $L\subseteq X^+$, we denote by \tcl L the topological closure of $L$ in~\Om
XS.

Given \um-terms $t_1,\ldots,t_m$ and languages~$L_1,\ldots,L_m$, we say
that $(t_1,\ldots,t_m)$ is a \emph{$(L_1,\ldots,L_m)$-splitting} of
a \um-term $t$ if the following
conditions hold:
\begin{enumerate}[$(i)$]
\item\label{item:local-adjustment-1} 
  $t=t_1\cdots t_m$;
\item\label{item:local-adjustment-2} 
  $\eval {t_i}S\in\tcl {L_i}$ for every $i=1,\ldots,m$.
\end{enumerate}
Given a $\sigma$-term $t$, let  $\lambda_\sigma(t)=(\rks\sigma
t,\nu_\sigma(t))$. We may write $\lambda$ instead of~$\lambda_\sigma$ when
$\sigma$ is clear from the context.

\begin{Lemma}
  \label{l:local-adjustment}
  Let $\sigma$ be a unary signature containing~$\kappa$, let $t$ be a $\bar
  \sigma$-term, and let $L_1,\ldots,L_m$ be rational languages. If $t$
  admits an $(L_1,\ldots,L_m)$-splitting $(t_1,\ldots,t_m)$, then there
  exists a $\bar
  \sigma$-term $z$  admitting an $(L_1,\ldots,L_m)$-splitting $(z_1,\ldots,z_m)$ such
  that:
  \begin{enumerate}[$(1)$]
  \item\label{i1} $\eval z S=\eval t S$,
  \item\label{i2} $z_i$ is a $\bar\sigma$-term  for $i=1,\ldots,m$, 
  \item\label{i3} $\lambda_{\bar\sigma}(z_i)=\lambda_\um(t_i)$ for  $i=1,\ldots,m$.
  \end{enumerate}
\end{Lemma}

\begin{proof}
  We only prove the statement for $m=2$, since it is representative of
  the general case, and it allows a simplified notation.

  Let $(t_1,t_2)$ be an $(L_1,L_2)$-splitting of $t$. Set
  $x_n=\Approx_n(t_1)$ and $y_n=\Approx_n(t_2)$. 
  By~\ref{item:local-adjustment-2}, $\lim x_n=\eval {t_1}S$
  belongs to \tcl {L_1}, which is open by rationality of~$L_1$
  \cite[Thm.~3.6.1]{Almeida:1994a}. Therefore, the word $x_n$~belongs
  to~$L_1$ for all sufficiently large~$n$. Similarly, the word $y_n$
  belongs to $L_2$ for all sufficiently large $n$.  Let $\varphi:X^+\to
  S$ be a homomorphism into a finite semigroup recognizing both
  languages $L_1$
  and~$L_2$.      
  In view of~\ref{item:local-adjustment-1}, $(x_n,y_n)_n$ is a
  factorization sequence for~$t$. Let $(x_{n_r},y_{n_r})_r$ be a
  filtered subsequence of $(x_n,y_n)_n$, and let
  $(k_{1,r},\ell_{1,r},\ldots,k_{d,r},\ell_{d,r})_r$ be the sequence of
  exponent vectors for the factorization
  $\Approx_{n_r}(t)=x_{n_r}y_{n_r}$.  When $(k_{i,r})_r$ (resp.\
  $(\ell_{i,r})_r$) is constant, let $k_i$ (resp.\ $\ell_i$) be this
  constant value. Otherwise, by taking a subsequence, we may 
  assume that for each $s\in S$, each of the sequences $(s^{k_{i,r}})_r$
  and $(s^{\ell_{i,r}})_r$ is constant, say with value respectively
  $s^{k_i}$ and $s^{\ell_i}$ ($i=1,\ldots,d$), the integers $k_i$ and
  $\ell_i$ being independent of the element $s\in S$.

  In view of Case~\eqref{item:factorization-2} of the definition of
  factorization sequence and since $t$ is a $\bar\sigma$-term, each
  sequence $(k_{i,r}+\ell_{i,r}+1)_r$ converges to some
  $\bar\sigma$-exponent~$\gamma_i$. In particular, $\gamma_i$ is infinite. Define $(\alpha_i,\beta_i)$ by
  \begin{equation*}
    (\alpha_i,\beta_i) =
    \begin{cases}
      (k_i,\gamma_i-k_i-1)&\text{if $(k_{i,r})_r$ is constant and
        $(\ell_{i,r})_r$ unbounded,}\\
      (\gamma_i-\ell_i-1,\ell_i)&\text{if $(k_{i,r})_r$ is unbounded and
        $(\ell_{i,r})_r$ constant,}\\
      (\gamma_i-\ell_i-1,\omega+\ell_i)\!\!&\text{if both $(k_{i,r})_r$ and
        $(\ell_{i,r})_r$ are unbounded.}
    \end{cases}
  \end{equation*}
  Note that in all cases, we have
  \begin{equation}
    \label{eq:invariant}
    \alpha_i+\beta_i+1=\gamma_i.
  \end{equation}
  Let $z_1$ (resp.\ $z_2$) be the \um-term obtained from $x_{n_r}$ (resp.\ from $y_{n_r}$) by
  replacing for every $i$ the exponent $k_{i,r}$ by $\alpha_i$ and the exponent $\ell_{i,r}$
  by $\beta_i$.  Set $$z=z_1z_2,$$ and let us verify that $z_1,z_2$ and $z$ fulfill the
  desired properties. We have to show properties~\ref{i1}--\ref{i3}, and that $(z_1,z_2)$ is
  an $(L_1,L_2)$-splitting of $z$.

  \smallskip First note that, by~\eqref{eq:invariant}, we have
  $\eval{y^{\alpha_i+\beta_i+1}}S=\eval{y^{\gamma_i}}S$ for all \um-term $y$.  Since
  $(k_{i,r}+\ell_{i,r}+1)_r$ converges to $\gamma_i$, using~\eqref{item:factorization-2} we
  deduce that $\eval z S=\eval t S$, which proves~\ref{i1}.  Next, by definition
  $\alpha_i$ and $\beta_i$ either belong to $\mathbb{N}$, or are of one of the forms
  $\gamma_i-n$ or $\omega+n$ where $n\in\mathbb{N}$.  Since $\gamma_i$ is a
  $\bar\sigma$-exponent and since $\sigma$ contains $\kappa$, both $\alpha_i$ and $\beta_i$
  are $\bar\sigma$-exponents, whence both $z_1$, $z_2$ are $\bar\sigma$-terms, which
  proves~\ref{i2}. Finally, we have $\lambda_{\bar\sigma}(z_i)=\lambda_\um (t_i)$ by
  construction, which is~\ref{i3}.

  It remains to verify that $(z_1,z_2)$ is an $(L_1,L_2)$-splitting of
  $z$. Condition~\ref{item:local-adjustment-1} is satisfied by definition of $z$. Let us verify that
  $z_1\in \tcl{L_1}$ (showing that $z_2\in\tcl {L_2}$ is similar). Let $\hat\varphi:\Om XS\to S$ be the
  continuous extension of $\varphi$ to $\Om XS$. 
  By~\ref{item:local-adjustment-2} applied to the
$(L_1,L_2)$-splitting $(t_1,t_2)$ of $t$, we have  $t_1\in\tcl{L_1}=\hat\varphi^{-1}(\varphi(L_1))$. Since $t_1$ is the
  limit of $(x_{n_r})_r$, it suffices to show that
  $\hat\varphi(z_1)=\varphi(x_{n_r})$. This follows
  from the claim that for $r$ large enough,
  $s^{\Approx_{n_r}(\alpha_i)}=s^{k_i}=s^{k_{i_r}}$, which is
  clear if $(k_{i_r})_r$ is constant, while it is obtained by reasoning in the
  group $\{s^{\omega+p}\mid p\ge0\}$ if $(k_{i_r})_r$ is unbounded.
\end{proof}

\subsection{The case of concatenation}
\label{sec:concatenation}
We are now ready to treat the case of concatenation, that is, to establish
Property~\eqref{eq:PR-product}.

\begin{Thm}
  \label{t:half-PR-for-S-1}
  Equality~\eqref{eq:PR-product} holds over the pseudovariety \pv S for every unary signature
  $\sigma$ containing $\kappa$ and for all rational languages $K,L\subseteq X^+$.
\end{Thm}

\begin{proof}
  The inclusion from right to left in~\eqref{eq:PR-product}
  amounts to continuity of multiplication in~\oms XS and thus it is
  always valid. For the direct inclusion, let $v$ be an arbitrary
  element of~$\tcls{KL}$. We need to show that $v$ belongs to~$\tcls
  K\cdot\tcls L$.

  Choose a $\sigma$-term $t$ such that $\eval t S=v$. Since $v\in\tcl{KL}$ and since the
  closure $\tcl{KL}$ of the rational language~$KL$ is
  clopen~\cite[Thm.~3.6.1]{Almeida:1994a}, the word $\Approx_n(t)$~belongs to~$KL$ for all
  sufficiently large $n$. For such~$n$, let $t_{1,n}\in K$ and $t_{2,n}\in L$ be words such
  that $\Approx_n(t)=t_{1,n}t_{2,n}$, and let $(t_{1,n_r},t_{2,n_r})_r$ be a filtered
  subsequence of $(t_{1,n},t_{2,n})_n$. For $i=1,2$, let $t_i$ be the term obtained by
  substituting each exponent vector with the limit of the sequence of exponent vectors, in
  $\widehat{\mathbb{N}}^d$, so that $\lim t_{i,n_r}=\eval{t_i}S$, and $(t_1,t_2)$ is a
  $(K,L)$-splitting of $t$.  By Lemma~\ref{l:local-adjustment}, it follows that there exists
  a $\bar\sigma$-term $z$ such that $v=\eval zS$ and $z$ admits a $(K,L)$-splitting $(z_1,
  z_2)$ into $\bar\sigma$-terms. Since the unary signature $\sigma$ contains $\kappa$, we
  have $\omup{\bar\sigma} XS=\oms XS$ by Remark~\ref{rem:inclusive}. Hence, $\eval
  {z_1}S\in\tcl{K}\cap\omup{\bar\sigma} XS=\tcl{K}\cap\oms XS=\tcls K$, and similarly $\eval
  {z_2}S\in\tcls L$.  Finally, $v=\eval zS=\eval {z_1z_2}S=\eval {z_1}S\cdot \eval {z_2}
  S\in\tcls K\cdot\tcls L$.
\end{proof}

\subsection{The case of iteration}
\label{sec:iteration}

We now show that~\eqref{eq:PR-plus} holds over the pseudovariety $\pv S$, for every pure implicit
signature $\sigma$ containing $\kappa$ and every rational language $L\subseteq X^+$. It is easy to
see that the inclusion from right to left in~\eqref{eq:PR-plus} always
holds, see~\cite{ACZ:Closures:14}.\label{easy-inclusion-thm-3.5} The rest of this subsection is devoted to the proof of the other
inclusion.
\begin{Thm}
  \label{t:half-PR-for-S-2}

  Equality~\eqref{eq:PR-plus} holds over the pseudovariety \pv S for every pure unary signature
  $\sigma$ containing $\kappa$ and for every rational language $L\subseteq X^+$.
\end{Thm}

The proof of Theorem~\ref{t:half-PR-for-S-2} follows the lines of its analog for
the pseudovariety \pv A which is presented
in~\cite[Section~6]{ACZ:Closures:14}, even though the
argument requires significant changes in several points.

Consider an element $v$ of~\tcls{L^+}. We must show that
there is a $\sigma$-term which coincides with~$v$ when evaluated on
(finitely many) suitable elements of~\tcls L. It turns out to be
convenient to assume more generally that $v\in\tclin{\bar\sigma}{L^+}$,
so that there exists a $\bar\sigma$-term $t$ such that $\eval t
S=v$. Therefore, we want to show that, for every $\bar\sigma$-term~$t$,
\begin{equation}
  \label{eq:inclusion-iteration}
\parbox{.8\textwidth}{for every rational
language $L$, $\eval t S\in\tcls{L^+}$ implies
$\eval tS\in\acls{\tcls L}$.}\tag{$\mathcal{P}_t$}
\end{equation}
Let $w_k=\Approx_k(t)$. The sequence of words $(w_k)_k$ converges
to~$v=\eval tS$ in~\Om XS. As $v$ belongs to the
open set \tcls{L^+}, the word $w_k$ belongs to~$L^+$ for all
sufficiently large~$k$, and we may therefore assume that there are
factorizations
\begin{equation}
  \label{eq:factorization-t}
   w_k=w_{1,k}\cdots w_{r_k,k},
\end{equation}
with each $w_{i,j}\in L$. If there is a bounded subsequence of the
sequence $(r_k)_k$, which counts the number of factors from $L$, then
Theorem~\ref{t:half-PR-for-S-1} yields that~$v$ belongs to the
subsemigroup of~\oms XS generated by \tcls L and we are done. We may
therefore assume that $\lim r_k=\infty$, which implies  that $\rk{t}\geq1$.
We first reduce the problem to the case~$\nu(t)=1$.

\begin{Prop}
  \label{prop:reduction-nu=1}
  Assume that~\eqref{eq:inclusion-iteration} holds for every
  $\bar\sigma$-term $t$ with $\nu(t)=1$.
  Then, it holds for every
  $\bar\sigma$-term $t$.
\end{Prop}

\begin{proof}
  Let $t$ be a $\bar\sigma$-term such that $\nu(t)>1$. Let
  $v=\eval tS$, and assume that $v\in\tcls{L^+}$. To show
  that $v\in\acls{\tcls{L}}$, we proceed by induction on
  $\lambda(t)$, for the lexicographical order on $\mathbb{N}\times\mathbb{N}$. Consider the factorization of~$t$ in $\bar\sigma$-terms as
  in~\eqref{eq:t} and the factorization~\eqref{eq:factorization-t} of
  $w_k=\Approx_k(t)$. We distinguish three cases.

  \smallskip
  \paragraph{\textbf{Case 1}} Suppose first that there are infinitely many indices~$k$
  for which there exists $i_k\in\{1,\ldots,r_{k}\}$ such that the first
  letter of the simplified history
  $$\sh_{k}(t,w_{1,k}\cdots
  w_{i_k,k},w_{i_k+1,k}\cdots   w_{r_{k},k})$$ is of one of the forms $(1,j)$ with $0\neq j\neq m$, 
  or $(2,j)$ with $1\neq j\neq m$. That is, the corresponding factorizations
  $(x_k,y_k)$, where $x_k=w_{1,k}\cdots w_{i_k,k}$ and
  $y_k=w_{i_k+1,k}\cdots w_{r_{k},k}$, do not split $\Approx_k(t)$ in
  its prefix $\Approx_k(t^{}_0s_1^{\alpha_1})$, nor in its suffix
  $\Approx_k(s_m^{\alpha_m}t^{}_m)$.
   
  By Lemma~\ref{lem:converging-factorization-sequences}, both $(x_k)_k$ and $(y_k)_k$ admit
  subsequences converging to \um-words, say $v_1=\eval{u_1}S$ and $v_2=\eval{u_2}S$ respectively,
  with $u_1u_2=t$. Since both $x_k$ and $y_k$ belong to $L^+$, we deduce that
  $v_1,v_2\in\tcl{L^+}$. Therefore, one can apply Lemma~\ref{l:local-adjustment}: there exist
  $\bar\sigma$-terms $z_1$ and $z_2$ such that $v=\eval{z_1z_2}S$, and for $i=1,2$,
  $\lambda(z_i)=\lambda(u_i)$ and $\eval{z_i}S\in\tcl{L^+}$. By Remark~\ref{rem:inclusive}, we
  obtain $\eval{z_i}S\in\tcls{L^+}$. By the assumption on the first letter of the simplified
  histories, we have $\rk{u_i}=\rk t$ and $\nu(u_i)<\nu(t)$, hence
  $\lambda(z_i)=\lambda(u_i)<\lambda(t)$ ($i=1,2$).  Arguing inductively, we deduce that
  $\eval{z_1}S$ and $\eval{z_2}S$ belong to $\acls{\tcls{L}}$, whence so does
  $v=\eval{z_1z_2}S=\eval{z_1}S\cdot\eval{z_2}S$. This concludes the proof for Case~1.

  \medskip\paragraph{\textbf{Case 2}}
  Assume now that for all sufficiently large $k$, there is an index
  $i_k$ such that the first letters of the simplified histories
  \begin{align*}
    &\sh_k(t,w_{1,k}\cdots w_{i_k-1,k},w_{i_k,k}\cdots w_{r_k,k})\\
    &\sh_k(t,w_{1,k}\cdots w_{i_k,k},w_{i_k+1,k}\cdots w_{r_k,k})
  \end{align*}
  are of the forms (1,0) or (2,1) for the first one,  and $(1,m)$ or $(2,m)$ for
  the second one. In other words, the factor $w_{i_k,k}$ of $\Approx_k(t)$ jumps from the
  prefix $\Approx_k(t^{}_0s_1^{\alpha_1})$ to the suffix
  $\Approx_k(s_m^{\alpha_m}t^{}_m)$.

  As in the first case, we may apply twice
  Lemma~\ref{lem:converging-factorization-sequences} to obtain
  from the following sequence of factorizations
  \begin{displaymath}
    w_k=w_{1,k}\cdots w_{i_k-1,k}\cdot w_{i_k,k} \cdot w_{i_k+1,k}\cdots w_{r_k,k}\qquad(k\ge1),
  \end{displaymath}
  an $(L^+,L,L^+)$-splitting of~$t$ into \um-terms, say $t=u_1u_2u_3$.  Applying
  Lemma~\ref{l:local-adjustment}, we deduce that there exists an $(L^+,L,L^+)$-splitting
  $(z_1,z_2,z_3)$ into $\bar\sigma$-terms of a $\bar\sigma$-term $z$ such that $\eval zS=v$
  and $\lambda(z_i)=\lambda(u_i)$, $i=1,2,3$.  By Remark~\ref{rem:inclusive}, we obtain
  $\eval {z_1}S,\eval {z_3}S\in\tcls{L^+}$ and $\eval {z_2}S\in\tcls{L}$.  By the hypothesis
  of Case 2, we know that for $i=1,3$, we have either $\rk{u_i}=\rk t$ and $\nu(u_i)=1$, or
  $\rk{u_i}<\rk t$. Hence, $\lambda(z_i)<\lambda(t)$.  Thus, we may apply the induction
  hypothesis to deduce that $\eval{z_1}S$ and $\eval{z_3}S$ belong to $\acls{\tcls
    L}$. Hence, we finally have $v=\eval{z_1}S\cdot\eval{z_2}S\cdot\eval{z_3}S\in\acls{\tcls{L}}\cdot\tcls{L}\cdot\acls{\tcls{L}}\subseteq\acls{\tcls{L}}$.

  \medskip\paragraph{\textbf{Case 3}}  
  Assume finally that for all sufficiently large $k$ and for all indices
  $i_k$, the first letter of the simplified history
  $\sh_k(t,w_{1,k}\cdots w_{i_k-1,k},w_{i_k,k}\cdots w_{r_k,k})$~is
  \begin{enumerate}[$(a)$]
  \item\label{item:1} either of the form (1,0) or (2,1), which means that $w_{r_k,k}$
    spans from the prefix $\Approx_k(t^{}_0s_1^{\alpha_1})$ to the end
    of  $\Approx_k(t)$,
  \item\label{item:2} or of the form $(1,n)$ or $(2,n)$, which means that $w_{1,k}$ jumps from the beginning of
    $\Approx_k(t)$ to the suffix $\Approx_k(s_m^{\alpha_m}t^{}_m)$.
  \end{enumerate}
  This case is treated as Case 2, setting $v_3$  (resp.\ $v_1$) to be the empty term,
  in case \ref{item:1} or  \ref{item:2} occurs.
\end{proof}

To conclude the proof of Theorem~\ref{t:half-PR-for-S-2}, it remains to
treat the case where $\nu(t)=1$. For dealing with this case, we use
directed weighted multigraphs.  A \emph{(multi)graph} is a tuple
$(Q,(E_{p,q})_{(p,q)\in Q\times Q})$ where $Q$ is a set of vertices, and
$E_{p,q}$ is a set of edges having source $p$ and target $q$, for each
pair of vertices $(p,q)\in Q\times Q$. In the sequel, the graphs shall
always be finite. A \emph{weighted multigraph} is given by a multigraph
together with a weight function, which associates to each edge $e$ a
nonnegative integer $\weight(e)$. If $e$ is an edge with source $p$ and
target $q$, we represent this edge by $$p\xrightarrow{\weight(e)}q.$$

For a path $\gamma$ of a graph $\Gamma$, let $c(\gamma)$ denote the
edge-induced subgraph of $\Gamma$ whose edges are those traversed
by~$\gamma$. We call $c(\gamma)$ the \emph{support}
of~$\gamma$. Furthermore, if~$\zeta$ is an edge of~$\Gamma$, then
$|\gamma|_\zeta$ denotes the number of times $\gamma$ goes through the
edge~$\zeta$. For a subgraph $\Gamma'$ of~$\Gamma$, we denote by
$|\gamma|_{\Gamma'}$ the \emph{minimum} of $|\gamma|_\zeta$ with $\zeta$ an arbitrary
edge of $\Gamma'$.

\begin{Lemma}
  \label{l:infinite-capacity-cycle}
  Let $(\pi_k)_k$ be a sequence of paths of a finite multigraph
  $\Gamma$. If there is some edge $\zeta$ for which the sequence
  $(|\pi_k|_\zeta)_k$ is unbounded, then there is some cycle $\gamma$
  such that $(|\pi_k|_{c(\gamma)})_k$ is unbounded.
\end{Lemma}

\begin{proof}
  Consider on each path $\pi_k$ the subpaths which start with the
  edge~$\zeta$ and whose length is the total number of
  vertices of the graph~$\Gamma$. Since there are only
  finitely many such subpaths, at least one of them, say $\delta$,
  must be used an unbounded number of times. Because $\delta$ must go
  at least twice through the same vertex, $\delta$ contains some cycle
  which satisfies the required condition.
\end{proof}

\noindent
We conclude the proof of Theorem~\ref{t:half-PR-for-S-2} by establishing
the following result, which, combined with
Proposition~\ref{prop:reduction-nu=1}, implies that
Property~\eqref{eq:inclusion-iteration} holds for every
$\bar\sigma$-term~$t$.
\begin{Prop}
  \label{prop:nu=1}
  Property~\eqref{eq:inclusion-iteration} holds for every $\bar\sigma$-term $t$ with $\nu(t)=1$.
\end{Prop}

\begin{proof}
  Let $t$ be a $\bar\sigma$-term with $\nu(t)=1$. Let $w=\eval tS$. Assuming that
  $w\in\tcls{L^+}$, we want to show that $w\in\acls{\tcls{L}}$. We have
  $t=t^{}_0s_1^{\alpha_1}t^{}_1$, with $\rk{t_0},\rk{t_1}\leq\rk{s_1}=\rk t-1$. Let
  $w_k=\Approx_k(t)$. Since for $k$ large enough, we have $w_k\in L^+$, one may consider a
  factorization~\eqref{eq:factorization-t} of~$w_k$.  Using a similar argument as in the
  proof of Case~2 of Proposition~\ref{prop:reduction-nu=1}, we may assume, replacing
  $(w_k)_k$ by a subsequence if necessary, that $\Approx_{k}(t_0)$ is a prefix of~$w_{1,k}$
  and $\Approx_{k}(t_1)$ is a suffix of~$w_{r_{k},k}$.
  
  Since $L$ is a rational language of~$X^+$, there is a homomorphism
  $\varphi:X^*\to M$ onto a finite monoid $M$ such that
  $\varphi^{-1}(1)=\{1\}$ and $\varphi^{-1}(\varphi(L))=L$. Let $m$ and
  $p$ be positive integers such that
  \begin{equation}
    \label{eq:half-PR-for-S-2-monoid-index-period}
    a^{m+p}=a^m \text{ for every $a\in M$.}
  \end{equation}
  We construct for each $k$ a finite directed multigraph $\Gamma_k$.
  The set of vertices~is
  \begin{equation*}
    V_k=\bigl\{(a,b)\in M\times M:
    \Approx_k(s_1)\in\varphi^{-1}(a)L^*\varphi^{-1}(b)\bigr\}
    \cup\{\initial,\final\},
  \end{equation*}
  where the two symbols $\initial$ and ${\final}$ do not belong to~$M$.
  The following are the edges of the graph~$\Gamma_k$, where $e$ denotes
  a natural number that does not exceed $\Approx_k(\alpha_1)$:
  \begin{itemize}
  \item there is an edge $(a_1,b_1)\xrightarrow{e+1}(a_2,b_2)$ if
    $L\cap\varphi^{-1}(b_1)\Approx_k(s_1^e)\varphi^{-1}(a_2)\neq\emptyset$;
  \item there is an edge $\initial\xrightarrow e(a,b)$ if
    $L\cap \Approx_k(t_0s_1^e)\varphi^{-1}(a)\neq\emptyset$;
  \item there is an edge $(a,b)\xrightarrow{e+1}{\final}$ 
    if $L\cap\varphi^{-1}(b)\Approx_k(s_1^et_1)\neq\emptyset$.
  \end{itemize}
  Observe that, in view of~\eqref{eq:half-PR-for-S-2-monoid-index-period}, there is an edge
  in the graph of the form $q_1\xrightarrow e q_2$ with $e\geq m$ if
  and only if there is also an edge $q_1\xrightarrow{e+p}q_2$ and $e+p\le\Approx_k(\alpha_1)$.
  
  The purpose of this graph is to capture factorizations of the product
  $\Approx_k(t_0)\Approx_k(s_1)^{\Approx_k(\alpha_1)}\Approx_k(t_1)$ belonging
  to $L^+$. More precisely, for each $k$, the factorizations
  \begin{equation}
    \label{eq:half-PR-for-S-2-factorizations}
    \begin{split}
      w_k %
      &=\Approx_k(t_0)\Approx_k(s_1)^{\Approx_k(\alpha_1)}\Approx_k(t_1) \\
      &=w_{1,k}\cdots w_{r_k,k}
    \end{split}
  \end{equation}
  determine a path $\pi_k$ from vertex $\initial$ to vertex~${\final}$:
  the factors $w_{i,k}$ which are not completely contained in
  some factor $\Approx_k(s_1)$ determine the edges. Each
  intermediate vertex in the path corresponds to a factor
  $\Approx_k(s_1)$ together with a factorization into a word,
  followed by a possibly empty product of elements from~$L$, followed
  by a word, where only the values under $\varphi$ of the prefix and
  suffix words are relevant.

  Conversely, every path $\gamma$ from $\initial$ to ${\final}$ determines
  a factorization of a word of the form $\Approx_k(t_0s_1^\ell
  t_1)$ into a product of elements of~$L$. Indeed, we may choose for
  each intermediate vertex $q$ words $u_{q,k},v_{q,k}\in X^*$ and
  $z_{q,k}\in L^*$ such that
  \begin{equation}
    \label{eq:z}
    \Approx_k(s_1)=u_{q,k}z_{q,k}v_{q,k}.
  \end{equation}
  Then, for each edge
  $\zeta:q_\zeta\xrightarrow{e+1}q_\zeta'$, the word
  \begin{equation}
    \label{eq:w_zeta}
    y_{\zeta,k}=v_{q_\zeta,k}\,\Approx_k(s_1^e)\,u_{q_\zeta',k}
  \end{equation}
  belongs to~$L$. If the path $\gamma$~is the sequence of edges
  $(\zeta_0,\zeta_1,\ldots,\zeta_r)$, with
  $\zeta_0:\initial\xrightarrow{e}q_{\zeta_0}'$ and
  $\zeta_r:q_{\zeta_r}\xrightarrow{e+1}{\final}$, then we also have words
  \begin{align*}
    y_{\zeta_0,k}&=\Approx_k(t_0s_1^e)\,u_{q_{\zeta_0}',k} \\
    y_{\zeta_r,k}&=v_{q_{\zeta_r},k}\,\Approx_k(s_1^et_1)
  \end{align*}
  in $L$. Then the following is the factorization associated with the
  path:
  \begin{equation}
    \label{eq:path-factorization}
    \Approx_k(t_0s_1^\ell t_1) %
    =y_{\zeta_0,k}\, %
    z_{q_{\zeta_1},k}\, %
    y_{\zeta_1,k} %
    \cdots %
    z_{q_{\zeta_{r}},k}\, %
    y_{\zeta_r,k}.
  \end{equation}
  The total number $\ell$ of factors $\Approx_k(s_1)$ that are covered by
  following the path~$\gamma$ is the sum of the weights of the edges, taking
  into account multiplicities; we call it the \emph{total weight} of the
  path. Combining with Euler's Theorem~\cite[Theorem~5.7.1]{Almeida:1994a}, it is now easy to deduce that
  each of the following transformations does not change the total weight
  of a path and therefore the value of the left side of the
  equality~\eqref{eq:path-factorization}: 
  \begin{enumerate}
  \item\label{item:key-1} to traverse the edges in a path in a
    different order, without changing the number of times we go
    through each edge;
  \item\label{item:key-2} suppose that in the support of the path
    there are two cycles $\delta_1$ and $\delta_2$, with respective
    total weights $n_1$ and $n_2$, and that the positive integers
    $r_1$ and $r_2$ are such that $n_1r_1=n_2r_2$; suppose further
    that the path goes through each edge in $\delta_i$ more than $r_i$
    times; to replace the path by another one which goes through each
    edge in $\delta_1$ less $r_1$ times than before and through each
    edge in $\delta_2$ more $r_2$ than before;
  \item\label{item:key-3} if there are two edges
    $q_1\xrightarrow{e_1}q_2$ and $q_3\xrightarrow{e_2+p}q_4$ in the
    path with both $e_1,e_2\geq m$, to replace in the path one
    occurrence of the edge $q_1\xrightarrow{e_1}q_2$ by that of the
    edge $q_1\xrightarrow{e_1+p}q_2$, provided we compensate by
    replacing one occurrence of the edge $q_3\xrightarrow{e_2+p}q_4$
    by $q_3\xrightarrow{e_2}q_4$;
  \item\label{item:key-4} suppose that in the support of the path
    there is a cycle $\delta$ with total weight $n$ and that the path
    goes through each edge in~$\delta$ at least $p+1$ times; suppose
    further that there is an edge $q_1\xrightarrow eq_2$ with weight
    at least $m$; replace the path by another one which goes through
    each edge in $\delta$ less $p$ times than before and change the
    edge $q_1\xrightarrow eq_2$ by $q_1\xrightarrow{e+np}q_2$.
  \end{enumerate}

  Using transformations of type~\ref{item:key-3}, we may assume that
  the path $\pi_k$ goes through at most one edge whose weight exceeds
  $m+p-1$. Therefore,  the remaining edges in $\pi_k$ are taken from a fixed finite
  set. Thus, by taking a subsequence, we may further assume that all
  paths $\pi_k$ use exactly the same edges of weight at most $m+p-1$
  and, either none of the $\pi_k$ use any other edges or, otherwise
  they all use only one additional edge $\zeta_k$ connecting  two fixed vertices.
  Hence all the graphs $c(\pi_k)$ are the same finite graph, up to an
  isomorphism that only changes one edge.

  On the other hand, using transformations of type~\ref{item:key-4}, we
  may assume that if all the paths $\pi_k$ go through some edge of
  weight at least $m$, then the graph $c(\pi_k)$ contains no cycle in
  which every edge is used at least $p+1$ times.
  
  We now split the argument into two cases. Suppose first that every
  $c(\pi_k)$ contains an edge of weight at least $m$.  In this case, one
  can apply Lemma~\ref{l:infinite-capacity-cycle} to deduce that there
  is a bound on the length of the paths $\pi_k$ and, therefore, we may
  assume that they all have the same length. Moreover, we may further
  assume that, except for the edge $\zeta_{i,k}$, at the same position~$i$,
  all paths
  $\pi_k=(\zeta_0,\ldots,\zeta_{i-1},\zeta_{i,k},\zeta_{i+1},\ldots,\zeta_r)$
  are identical. Consider the factorizations
  $$w_{k} %
  =y_{\zeta_0,k}\, %
  z_{q_{\zeta_1},k}\, %
  \cdots %
  y_{\zeta_{i-1},k}\, %
  z_{q_{\zeta_{i}},k}\, %
  \,y_{\zeta_{i,k},k}\, %
  z_{q_{\zeta_{i+1}},k}\, %
  y_{\zeta_{i+1},k}\, %
  \cdots %
  z_{q_{\zeta_{r}},k}\, %
  y_{\zeta_r,k}$$
  of the form~\eqref{eq:path-factorization} associated with each of
  the paths $\pi_k$.

  Let $e_j$ be the weight of each edge $\zeta_j$ ($j\neq i$) and let
  $e_{i,k}$ be the weight of the edge $\zeta_{i,k}$. Computing the
  total weight, we obtain the formula
  \begin{equation}
    \label{eq:total-weight_k-case-1}
    \Approx_{k}(\alpha_1)=e_{i,k}+\sum_{j\neq i}e_j.
  \end{equation}
  Letting $k\to\infty$ in~\eqref{eq:total-weight_k-case-1}, we deduce that $\lim
  e_{i,k}=\alpha_1-\sum_{j\neq i}e_j$ and, therefore, $e_i=\lim e_{i,k}$ is a $\bar\sigma$-exponent,
  since so is $\alpha_1$ and since by definition, $\bar\sigma$ is complete.

  According to~\eqref{eq:w_zeta}, the factors $y_{\zeta_j,k}$ with
  $j\notin\{0,i,r\}$ are given by
  $$y_{\zeta_j,k}=v_{q_{\zeta_j},k}\Approx_{k}(s_1^{e_j})u_{q_{\zeta_{j+1},k}}.$$
  By compactness of \Om XS, we may assume that each of the sequences
  $(u_{q,k})_k$, $(v_{q,k})_k$, and $(z_{q,k})_k$ converges to the
  respective limit $u_q$, $v_q$, and $z_q$
  ($q=q_{\zeta_1},\ldots,q_{\zeta_r}$). Let
  \begin{align*}
    y_0&=\eval{t_0s_1^{e_0}}Su_{q_{\zeta_1}}\\
    y_r&=v_{q_{\zeta_r}}\eval{s_1^{e_r}t_1}S\\
    y_j&=v_{q_{\zeta_j}}\eval{s_1^{e_j}}Su_{q_{\zeta_{j+1}}} \
    (j=1,\ldots,r-1).
  \end{align*}
  Then we obtain a factorization
  $$w %
  =y_0\, %
  z_{q_{\zeta_1}}\, %
  \cdots %
  y_{i-1}\, %
  z_{q_{\zeta_{i}}}\, %
  \,y_i\, %
  z_{q_{\zeta_{i+1}}}\, %
  y_{i+1}\, %
  \cdots %
  z_{q_{\zeta_{r}}}\, %
  y_r$$
  in which each $y_j$ belongs to $\tcl L$, while the $z_q$ belong to
  $\tcl{L^+}\cup\{1\}$. By Lemma~\ref{l:local-adjustment}, we may
  assume that each $u_q$, $v_q$, and $z_q$ is a $\bar\sigma$-word. By definition~\eqref{eq:z} of the words $z_{q,k}$, the
  latter has rank less than $\rk w$. It follows that so is each $y_j$.
  Hence the $y_j$'s belong to $\tcls L$ and the $z_q$ belong
  to~$\tcls{L^+}\cup\{1\}$. By the induction hypothesis, each $z_q$
  belongs to~$\acls{\tcls L}$. Hence, $w$ belongs to~$\acls{\tcls L}$,
  which completes the proof of the first case.

  \medskip
  It remains to consider the case where all edges have weight less
  than~$m$. By previous reductions, we know that the graph $c(\pi_k)$ is
  constant. Because the total weight tends to $\infty$, so does the
  length of the path $\pi_k$. By
  Lemma~\ref{l:infinite-capacity-cycle}, there is some simple cycle
  $\gamma$ for which the sequence $(|\pi_k|_{c(\gamma)})_k$ is
  unbounded. Applying transformations of type~\ref{item:key-2}, we may
  assume that there is only one such cycle. By
  Lemma~\ref{l:infinite-capacity-cycle}, we deduce that the paths
  $\delta$ with $c(\delta)\subseteq c(\pi_k)$ which go at most once
  through each edge in~$c(\gamma)$ have bounded length. Hence, using
  transformations of type~\ref{item:key-1}, we may assume that there
  is a path $\delta$ from $\initial$ to ${\final}$ 
  such that the path
  $\pi_k$ is obtained by inserting the power cycle $\gamma^{\ell_k}$
  at a fixed vertex in the path~$\delta$, say $\pi_k$ is the
  concatenated path $\delta_0\gamma^{\ell_k}\delta_1$. Let the total
  weights of the paths $\delta_i$ and $\gamma$ be respectively $n_i$
  and $n$. Then the total weight of the path $\pi_k$ is given by
  the formula
  \begin{equation}
    \label{eq:total-weight_k-case-2}
    \Approx_{k}(\alpha_1)=n_0+n_1+n\ell_k.
  \end{equation}
  By taking a subsequence, we may assume that the sequence
  $(\ell_k)_k$ converges to some $\beta\in\widehat{\mathbb{N}}$.
  From \eqref{eq:total-weight_k-case-2}, it follows that
  $n\beta=\alpha_1-n_0-n_1$. Since the signature $\sigma$ is assumed
  to be pure, we deduce that $\beta$ is a $\bar\sigma$-exponent.

  By the argument in the preceding case, using the induction
  hypothesis, each of the paths $\delta_i$ and $\gamma$ determines a
  corresponding element of~$\acls{\tcls L}$, say respectively $y_i$
  and $y$, such that $w=y_0y^\beta y_1$. Since $\beta$ is a
  $\bar\sigma$-exponent, we may now end the proof by observing that it
  follows that $w\in\acls{\tcls L}$.
\end{proof}

\section{Pin-Reutenauer versus fullness}
\label{sec:PR}

In this section, we apply the main results of Section~\ref{sec:closures}
to show that the Pin-Reutenauer procedure is valid for many
pseudovarieties. For this purpose, we establish relationships between
that property and fullness, a notion introduced
by~\cite{Almeida&Steinberg:2000a}. See
also~\cite{Almeida&Steinberg:2000b,ACZ:Closures:14} for related properties and
other applications of fullness.

Recall that \pj V denotes the natural continuous homomorphism from $\Om XS$ to $\Om XV$. The
pseudovariety \pv V is said to be \emph{full} with respect to a class \Cl C of subsets
of~\oms XS if the following equality holds for every $L\in\Cl C$:
\begin{equation}
  \label{eq:fullness}
  \pj V(\tcls L)=\tclsv V{\pj V(L)}.
\end{equation}
The closure of the left hand side of~\eqref{eq:fullness} is taken in \oms XS, while the
closure of the right hand side is taken in \oms XV.
We say that \pv V is \emph{(weakly) $\sigma$-full} if, for every
finite alphabet $X$, \pv V is full with respect to the set of all
rational languages of~$X^+$. We also say that \pv V is \emph{strongly
  $\sigma$-full} if, for every finite alphabet $X$, \pv V is full with
respect to the class of all rational subsets of~\oms XS.

\subsection{The general case}
\label{sec:general-case}

We first consider the case of arbitrary pseudovarieties and signatures.

\begin{Prop}
  \label{p:PR->full}
  Let $\sigma$ be an arbitrary implicit signature, \pv V be a
  pseudovariety, and \Cl C be the closure under the rational operations
  of some set of finite subsets of~\oms XS. If the Pin-Reutenauer
  procedure holds for~\Cl C, then \pv V is full with respect
  to~\Cl C.
\end{Prop}

\begin{proof}
  Let $L$ be an arbitrary member of~\Cl C. We need to show that the
  equality~\eqref{eq:fullness} holds. The inclusion from left to right
  is an immediate consequence of the continuity of the mapping~\pj V.
  For the reverse inclusion, we proceed by induction on the
  construction of a rational expression of~$L$ in terms of finite sets
  in~\Cl C. If $L$ is a finite set, then $\pj V(\tcls L)=\pj V(L)$ and $\tclsv
  V{\pj V(L)}=\pj V(L)$, and so the equality~\eqref{eq:fullness} is trivially
  verified. Suppose next that $L_1$ and $L_2$ are elements of~\Cl C
  for which the equality~\eqref{eq:fullness} holds. Since topological
  closure and the application of mappings commutes with union, the
  equality~\eqref{eq:fullness} also holds for $L=L_1\cup L_2$. On the
  other hand, we have the following equalities and inclusions:
  \begin{alignat*}{3}
    \tclsv V{\pj V(L_1L_2)} %
    &=\tclsv V{\pj V(L_1)}\cdot\tclsv V{\pj V(L_2)} %
    &\quad&\text{\parbox[t]{.4\textwidth}%
      {since the Pin-Reutenauer procedure holds for~\Cl
      C}}\\
    &=\pj V(\tcls{L_1})\cdot\pj V(\tcls{L_2})
    &&\text{by the induction hypothesis}\\
    &=\pj V(\tcls{L_1}\cdot\tcls{L_2})
    &&\text{since \pj V is a homomorphism}\\
    &\subseteq\pj V(\tcls{L_1L_2})
    &&\text{by continuity of multiplication.}
  \end{alignat*}
  Taking into account that \pj V is a homomorphism of
  $\sigma$-semigroups and that the inclusion from right to left
  in~\eqref{eq:PR-plus} always holds (see the paragraph preceding Theorem 3.5, p.~\pageref{easy-inclusion-thm-3.5}), one can similarly show that
  \eqref{eq:fullness} holds for $L=L_1^+$.
\end{proof}

The following is an immediate application of
Proposition~\ref{p:PR->full}.

\begin{Cor}
  \label{c:PR->full}
  Let $\sigma$ be an arbitrary implicit signature and \pv V a
  pseudovariety. If \pv V is (strongly) $\sigma$-PR then \pv V is
  (respectively strongly) $\sigma$-full.\qed
\end{Cor}

The weak version of the following result is proved
in~\cite[Proposition~3.2]{ACZ:Closures:14}. The same proof
applies to the strong case.

\begin{Prop}
  \label{p:hereditarity}
  Let \pv V and \pv W be two (strongly) $\sigma$-full pseudovarieties
  such that $\pv V\subseteq\pv W$. If \pv W is (respectively strongly)
  $\sigma$-PR, then so is \pv V.
\end{Prop}

Note that \pv S is trivially $\sigma$-full for every implicit signature $\sigma$.
Combining Theorem~\ref{t:main} with Corollary~\ref{c:PR->full} and
Proposition~\ref{p:hereditarity}, we obtain the following result.

\begin{Cor}
  \label{c:main}
  Let $\sigma$ be a pure unary signature containing $\kappa$. Then a
  pseudovariety \pv V is $\sigma$-PR if and only if it is
  $\sigma$-full.\qed
\end{Cor}

We do not know whether the hypothesis on the signature can be dropped
in Corollary~\ref{c:main}.

\subsection{The group case}
\label{sec:groups}

We now consider the case of pseudovarieties of groups, for the signature~$\kappa$.

Recall that a group is \emph{LERF (locally extendible residually finite)} if
every finitely generated subgroup is closed in the profinite topology. We say
that a pseudovariety of groups \pv H is \emph{LERF} if, for every finite
alphabet~$X$, the relatively free group \omc XH~is LERF. By a classical result
of~\cite{Hall:1950}, the pseudovariety \pv G of all finite
groups is LERF.

By \cite[Proposition~2.9]{Margolis&Sapir&Weil:1999}, \pv G~is in fact
the only nontrivial extension-closed pseudovariety of groups that is
LERF. On the other hand, it is easy to check that, for the pseudovariety
\pv{Ab} of all finite Abelian groups, every subgroup of $\omc X{Ab}$ is
closed~\cite[Proposition~3.8]{Delgado:Abelian-poinlikes-monoid:1998:a}.

A slightly different notion of strongly $\kappa$-PR pseudovariety was
considered by~\cite{Pin&Reutenauer:1991} and~\cite{Delgado:2001}
(where it is simply called PR). Instead of property
\eqref{eq:PR-plus}, the following property is considered:
\begin{equation}
  \label{eq:PR-plus-bis}
  \tclsv V{L^+} = \acls{L}.
\end{equation}
Compared to~\eqref{eq:PR-plus}, the topological closure in the right hand side
of~\eqref{eq:PR-plus-bis} has been dropped. As observed in~\cite[end of Section
4]{ACZ:Closures:14}, Equation~\eqref{eq:PR-plus-bis} fails for the pseudovariety \pv S and
the implicit signature~$\kappa$, for $L=a^+b^+$, since $a^\omega b\in \tclsv
V{L^+}\setminus\acls{L}$. However, the two notions coincide for the
pseudovariety~\pv G \cite[Theorem~2.4]{Pin&Reutenauer:1991}. With same
argument, we generalize this result to LERF pseudovarieties.

\begin{Lemma}
  \label{l:strongly-k-PR-vs-PR}
  Let \pv V be a pseudovariety. If\/ \pv V satisfies \eqref{eq:PR-plus-bis} for a subset $L$ of\/
  \oms XV, then \eqref{eq:PR-plus} also holds. If\/
  \pv V is a LERF pseudovariety of groups and $\sigma=\kappa$, then
  \eqref{eq:PR-plus-bis}~holds for rational subsets $L$ of~\omc XV.
\end{Lemma}

\begin{proof}
  Since for $L\subseteq\oms XV$, the inclusions %
  $\acls{L} \subseteq\acls{\tclsv VL} \subseteq\tclsv V{L^+}$ always
  hold, it is clear that \eqref{eq:PR-plus-bis} implies
  \eqref{eq:PR-plus}.

  For the second part, we argue as
  in~\cite[p.~4]{ACZ:Closures:14}. Let $L$ be a
  rational subset of~\omc XV. Then, 
  $\kacls {L}=(L\cup L^{-1})^+$ is rational in \omc XV.  By
  a well-known theorem of~\cite{Anissimow&Seifert:1975}, the subgroup $\kacls {L}$ is
  therefore finitely generated. Hence, $\kacls {L}$ is closed in~\omc
  XV, by the assumption that \pv V is a LERF pseudovariety. Finally,
  we have $L^+\subseteq \kacls {L}$, hence $\tclsv V{L^+}\subseteq
  \kacls {L}$, which, combined with the reverse inclusion, which
  always holds, yields the result.
\end{proof}

It can be shown easily that a $\kappa$-PR pseudovariety  of groups
is LERF (see~\cite{Delgado:1997i}).
 Thus, a pseudovariety of groups is strongly $\kappa$-PR if
and only if it is PR in the sense of~\cite{Delgado:2001}.
In \cite[Corollary~3.9]{Delgado:2001}, it is also
established that every ``weakly PR'' pseudovariety of groups is
$\kappa$-full, a result which is considerably improved in the present
paper, in the form of Corollaries~\ref{c:PR->full} and~\ref{c:main}.

It was conjectured by~\cite{Pin&Reutenauer:1991} that \pv G is
strongly $\kappa$-PR. Their conjecture was reduced to another
conjecture, namely that the product of finitely many finitely
generated subgroups of a free group is closed. The latter conjecture
was established by~\cite{Ribes&Zalesskii:1993a}. Combining with
Proposition~\ref{p:hereditarity} and Corollary~\ref{c:PR->full}, we
obtain the following result.

\begin{Thm}
  \label{t:Strongly-PR-vs-full-groups}
  A pseudovariety of groups is strongly $\kappa$-PR if and only if it
  is strongly $\kappa$-full.\qed
\end{Thm}

The diagram in \figurename~\ref{fig:groups} summarizes the
results of this subsection. 
We say that a pseudovariety of groups \pv H is \emph{strong RZ} if in every
finitely generated free \pv H-group, any finite product of finitely
generated subgroups is closed. We say that \pv H is \emph{weak RZ} if, in every
finitely generated free \pv H-group, any finite product of finitely
generated \emph{closed} subgroups is also closed.

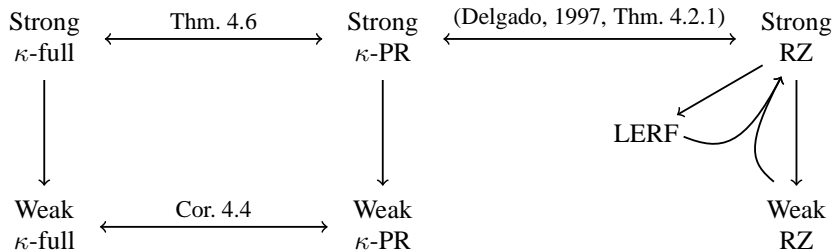
\begin{figure}[H]
  \centering
    \begin{tikzpicture}
      \node (sf) at (0,0) {
        \begin{tabular}[h]{c}
          Strong\\ $\kappa$-full
        \end{tabular}
      }; \node (spr) at ($(sf)+(4.5,0)$) {
        \begin{tabular}[h]{c}
          Strong\\ $\kappa$-PR
        \end{tabular}
      }; \node (srz) at ($(spr)+(5.5,0)$) {
        \begin{tabular}[h]{c}
          Strong\\ RZ
        \end{tabular}
      };
      \node (hall) at ($(srz)+(-2,-1.25)$) {LERF};
      \node (wf) at
      ($(sf)+(0,-2.5)$) {
        \begin{tabular}[h]{c}
          Weak\\ $\kappa$-full
        \end{tabular}
      }; \node (wpr) at ($(wf)+(4.5,0)$) {
        \begin{tabular}[h]{c}
          Weak\\ $\kappa$-PR
        \end{tabular}
      }; \node (wrz) at ($(wpr)+(5.5,0)$) {
        \begin{tabular}[h]{c}
          Weak\\ RZ
        \end{tabular}
      };

      \draw[equiv] (spr) to node[above] {\small Thm.~\ref{t:Strongly-PR-vs-full-groups}} (sf); 
      \draw[equiv] (spr) to node[above] {\small\cite[Thm.~4.2.1]{Delgado:1997i}} (srz); 
      \draw[equiv] (wpr) to node[above] {\small Cor.~\ref{c:main}} (wf);
      \draw[implies] (spr) to (wpr); 
      \draw[implies] (sf) to (wf);
      \draw[implies] (srz) to (wrz);
      \draw[implies] ($(srz)+(-.45,-.35)$) to (hall);

      \draw [implies] ($(hall)-(-.5,.05)$) .. controls ($(hall)-(-1,.25)$)
      and ($(hall)-(-1.3,0.2)$) .. ($(srz)+(-0.2,-.5)$);

      \draw [implies] ($(wrz)+(-.3,.6)$) .. controls ($(hall)-(-1.4,.4)$)
      and ($(hall)-(-1.3,0.2)$) ..  ($(srz)+(-0.2,-.5)$);
    \end{tikzpicture}
  \caption{Summary of results: pseudovarieties of groups}
\label{fig:groups}
\end{figure}

\bibliographystyle{abbrvnat}

\end{document}